\documentclass[a4paper,12pt]{article}
\usepackage{amsfonts}
\usepackage{amssymb}
\usepackage{amsmath}
\usepackage{latexsym}
\usepackage{enumerate}
\usepackage{amsthm}
\newtheorem{lem}{Lemma}[section]

\newtheorem{thm}{Theorem}[section]

\usepackage[usenames,dvipsnames]{pstricks}
\usepackage{epsfig}
\usepackage{pst-grad}
\usepackage{pst-plot}
\usepackage{txfonts}
\usepackage{lscape}

\begin{document}
\title{Decomposition of Cartesian Product of Complete Graphs into Sunlet Graphs of Order Eight}
\author{{K. Sowndhariya and A. Muthusamy} \\ 
{\footnotesize Department
of Mathematics, Periyar University, Salem,} \\ 
{\footnotesize
Tamil Nadu, India.} \\ {\footnotesize \emph{ email:sowndhariyak@gmail.com; ambdu@yahoo.com}}}
\date{}
\maketitle
\begin{abstract}
For any integer $k\geq 3$ , we define the sunlet graph of order $2k$, denoted by $L_{2k}$, as the graph consisting of a cycle of length $k$ together with $k$ pendant vertices such that, each pendant vertex adjacent to exactly one vertex of the cycle so that the degree of each vertex in the cycle is $3$. In this paper, we establish necessary and sufficient conditions for the existence of decomposition of the Cartesian product of complete graphs into sunlet graph of order eight.
\end{abstract}
\section{Introduction}
All graphs considered here are finite, simple and undirected. A \textit{cycle} of length $k$ is called $k$-cycle and it is denoted by $C_{k}$. $K_{m}$ denotes the \textit{complete} graph on $m$ vertices and $K_{m,n}$ denotes the \textit{complete bipartite} graph with $m$ and $n$ vertices in the parts. We denote the \textit{complete $m$-partite} graph with $n_{1}, n_{2}, \dots , n_{m}$ vertices in the parts by $K_{n_{1}, n_{2}, \dots , n_{m}}$. For any integer $\lambda > 0$, $\lambda G$ denotes the graph consisting of $\lambda$ edge-disjoint copies of $G$.
\par For any two graphs $G$ and $H$ of orders $m$ and $n$, respectively, the \textit{corona product} $G \odot H$ is the graph obtained by taking one copy of $G$ and $m$ copies of $H$ such that the $i$th vertex of $G$ is connected to every vertex in the $i$th copy of $H$. We define the \textit{sunlet graph} $L_{2k}$ with  $V(L_{2k})$ $=$ $\left\{x_{1}, x_{2}, \dots , x_{k}, x_{k+1}, x_{k+2}, \dots , x_{2k}\right\}$ and $E(L_{2k})$ $=$ $\left\{x_{i}x_{i+1} \cup x_{i}x_{k+i} \ | \ i = 1,2,. . ., k \ \textnormal{and subscripts of the first term is taken addition  modulo} \ k\right\}$. We denote it by  
\begin{center} $L_{2k}$ $=$ $\begin{pmatrix}x_{1} & x_{2} & \dots & x_{k} \\ x_{k+1} & x_{k+2} & \dots & x_{2k}\end{pmatrix}$. \end{center} Clearly, $L_{2k} \cong C_{k} \odot K_{1}$. 
\par The \textit{Cartesian product} of two graphs, $G$ and $H$, denotes by $G \Box H$, has the vertex set $V(G) \times V(H)$ and in which two vertices $(g, h)$ and $(g', h')$ are adjacent if and only if either $g = g'$ and $h$ is adjacent to $h'$ in $H$ or $h = h'$ and $g$ is adjacent to $g'$ in $G$. It is well known that  Cartesian product is commutative, associative and distributive over edge-disjoint union of graphs. 
\par We shall use the following notation throughout the paper. Let $G$ and $H$ be simple graphs with vertex sets $V(G)$ $=$ $\left\{x_{1}, x_{2}, \dots , x_{n}\right\}$ and $V(H)$ $=$ $\left\{y_{1}, y_{2}, \dots , y_{m}\right\}$. Then for our convenience, we write $V(G) \times V(H)$ $=$ $\bigcup_{i = 1}^{n}$ $X_{i}$, where $X_{i}$ stands for $x_{i} \times V(H)$. Further, in the sequel, we shall denote the vertices of $X_{i}$ as $\left\{x_{i}^{j} | 1 \leq j \leq m\right\}$, where $x_{i}^{j}$ stands for the vertex $(x_{i}, y_{j}) \in V(G) \times V(H)$. 
\par By a \textit{decomposition} of a graph $G$, we mean a list of edge-disjoint subgraphs whose union is $G$. For a graph $G$, if $E(G)$ can be partitioned into $E_{1}, E_{2}, . . . , E_{k}$ such that the subgraph induced by $E_{i}$ is $H_{i}$, for all $i$, $1 \leq i \leq k$, then we say that $H_{1}, H_{2}, . . . , H_{k}$ decompose $G$ and we write $G$ $=$ $H_{1} \ \oplus \ H_{2} \ \oplus . . . \ \oplus \ H_{k}$, since $H_{1}, H_{2}, . . . , H_{k}$ are edge-disjoint subgraphs of $G$. For $1 \leq i \leq k$, if $H_{i}$ $=$ $H$, we say that $G$ has a \textit{$H$- decomposition}.
\par Study of $H$-decomposition of graphs is not new. Many authors around the world are working in the field of cycle decomposition, path decomposition, star decompositon, trail decomposition, Hamilton cycle decomposition problems in graphs. Here we consider the sunlet decomposition of product graphs. Anitha and Lekshmi $\cite{AL1, AL}$ proved that $n$-sun decomposition of complete graph, complete bipartite graph and the Harary graphs. Liang and Guo $\cite{LG, LGW}$ gave the existence spectrum of a $k$-sun system of order $v$ as $k = 2, 4, 5, 6, 8$. Fu et. al. $\cite{FJLS, FJLS1}$ obtained that $3$-sun decompositions of $K_{p,p,r}$, $K_{n}$and embed a cyclic steiner triple system of order $n$ into a $3$-sun system of order $2n-1$, for $n = 1\pmod6$. Further they obtained $k$-sun system when $k = 6, 10, 14, 2^{t},$ for $t>1$. Fu et. al. $\cite{FHL}$ obtained the existence of a $5$-sun system of order $v$. Gionfriddo et.al. \cite{GFMT} obtained the spectrum for uniformly resolvable decompositions of $K_{v}$ into $1$-factor and $h$-suns. Akwu and Ajayi $\cite{AA}$ obtained the necessary and sufficient conditions for the existence of decomposition of $K_{n} \otimes \overline{K_{m}}$ and $(K_{n} - I) \otimes \overline{K_{m}}$, where $I$ denote the $1$-factor of a complete graph into sunlet graph of order twice the prime. Sowndhariya and Muthusamy $\cite{SM}$ obtained necessary and sufficient conditions for the existence of decomposition of $K_{m} \times K_{n}$ and $K_{m} \otimes \overline{K_{n}}$ into sunlet graph of order eight. 
\par In this paper, we prove the existence of $L_{8}$-decomposition of $K_{m} \Box K_{n}$. In fact, we establish necessary and sufficient conditions for the existence of $L_{8}$-decomposition of $K_{m} \Box K_{n}$. To prove our results, we state the following:
\begin{thm}[\cite{SM}]\label{1.1} 
For $k \geq 3$, $K_{4k+1}$ has a $L_{2k}$- decomposition.
\end{thm}
\begin{thm}[\cite{SM}]\label{1.2} 
For an integer $p > 0$, $K_{16p+1}$ has a $L_{8}$- decomposition.
\end{thm} 
\begin{lem}[\cite{SM}]\label{1.3} 
There exists a $L_{8}$- decomposition of $K_{16}$. 
\end{lem}
\begin{thm}[\cite{SM}]\label{1.4}
For any $m,n \geq 4$, $K_{m,n}$ has a $L_{8}$- decomposition if and only if $mn \equiv 0 \ (mod \ 8)$ except $(m,n) = (8,5), (4,4r+2)$ where $r > 0$.
\end{thm}
\section{Decomposition of $K_{m} \Box K_{n}$ into Sunlet Graph of order Eight}
\subsection{Necessary Conditions}
\begin{lem}\label{2.1}
If $K_m \Box K_n$has an $L_{8}$-decomposition, then either
\begin{enumerate}
  \item $m, n \equiv \ 0 \ (mod \ 4)$ 
	\item $m \equiv \ 0 \ (mod \ 8), \ n \equiv \ 0 \ (mod \ 2)$
	\item $m \equiv \ 0 \ (mod \ 16)$
	\item $m \equiv \ 1 \ (mod \ 16), \ n \equiv \ 1 \ (mod \ 16)$ 
	\item $m \equiv \ 15 \ (mod \ 16), \ n \equiv \ 3 \ (mod \ 16)$
	\item $m \equiv \ 13 \ (mod \ 16), \ n \equiv \ 5 \ (mod \ 16)$
	\item $m \equiv \ 11 \ (mod \ 16), \ n \equiv \ 7 \ (mod \ 16)$
	\item $m \equiv \ 9 \ (mod \ 16), \ n \equiv \ 9 \ (mod \ 16)$
\end{enumerate}
\end{lem}
\begin{proof}
The graph $K_m \Box K_n$ has $mn$ vertices, each having degree $m+n-2$. Hence, $K_m \Box K_n$ has $\frac{mn(m+n-2)}{2}$ edges. Assume that $K_{m} \Box K_{n}$ admits an $L_{8}$- decomposition. Then the number of edges in the graph must be divisible by $8$. i.e., $16|mn(m+n-2)$. Hence these conditions are met in each of the above eight cases and only in these cases.
\end{proof}
\subsection{Sufficient Conditions}
We now prove the above necessary conditions are also sufficiency.
\begin{lem}\label{2.2}
If $m \equiv 0\pmod4$ and $n \equiv 0\pmod4$, then the graph $K_m \Box K_n$ has an $L_{8}$-decomposition.
\end{lem} 
\begin{proof}
Let $m = 4s$ and $n = 4t$ for some $s, t > 0$. We can divide the graph $K_m \Box K_n$ into $st\left(K_4 \Box K_4\right)$ and the remaining graph can be viewed as $2st(s+t-2)K_{4,4}$. The $L_8$-decomposition of $K_4 \Box K_4$ is shown in Fig 1. By Theorem $\ref{1.4}$, $K_{4,4}$ has an $L_8$-decomposititon. Hence the graph $K_m \Box K_n$ has the desired decomposition.
\end{proof}
\begin{figure}[h]
\centering
\includegraphics[width=.80\textwidth]{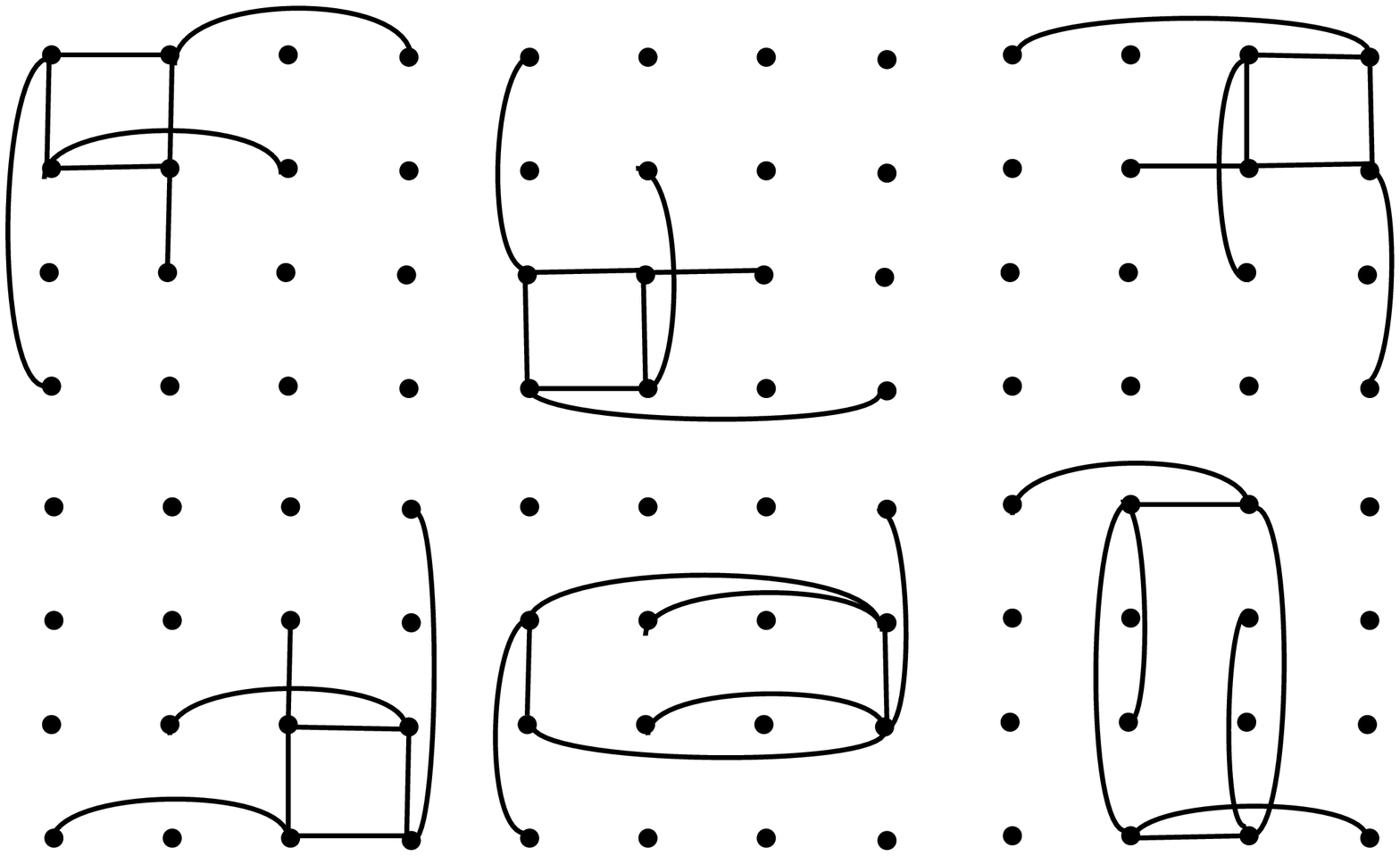}
\label{fig 1}
\caption{$L_{8}$-decomposition of $K_4 \Box K_{4}$.}
\end{figure}
\begin{lem}\label{2.3}
If $m \equiv 0\pmod8$ and $n \equiv 0\pmod2$, then the graph $K_m \Box K_n$ has an $L_{8}$-decomposition.
\end{lem} 
\begin{proof}
Let $m \equiv 0\pmod8$ and $n \equiv 0 \ (or) \ 4\pmod8$, then the proof follows from Lemma \ref{2.2}. Now let $m = 0\pmod8$ and consider two cases for $n$. \\
\textit{\textbf{Case (1)}} $n \equiv \ 2 \pmod8$. \\
Let $m = 8s$ and $n = 8t+2$ for some $s, t > 0$. Then we can write $K_{m} \Box K_{n}$ as $s(t-1)\left(K_{8} \Box K_{8}\right)$ $\oplus$ $s\left(K_{8} \Box K_{10}\right)$ $\oplus$ $\frac{s\left[(t-1)(t-2) + (8t+2)(s-1)\right]}{2}K_{8,8}$ $\oplus$ $8s(t-1)K_{8,10}$. An $L_{8}$-decomposition of $K_8 \Box K_{10}$ is presented in Appendix $4.1.1$ and the $L_8$-decomposition of the graphs $K_{8} \Box K_{8}$, $K_{8,8}$ and $K_{8,10}$ follows from Lemma \ref{2.2} and Theorem \ref{1.4}. \\
\textit{\textbf{Case (2)}} $n \equiv \ 6 \pmod8$. \\
Let $m = 8s$ and $n = 8t+6$ for some $s, t > 0$. Then we can write $K_{m} \Box K_{n}$ as $st\left(K_{8} \Box K_{8}\right)$ $\oplus$ $s\left(K_{8} \Box K_{6}\right)$ $\oplus$ $s\left[(s-1)(4t+3) + 4t(t-1)\right]K_{8,8}$ $\oplus$ $8stK_{8,6}$. An $L_{8}$-decomposition of $K_8 \Box K_{6}$ is appears in Appendix $4.1.2$ and the $L_8$-decomposition of the remaining graphs follows from Lemma \ref{2.2} and Theorem \ref{1.4}. Hence the graph $K_m \Box K_n$ has the desired decomposition.
\end{proof}
\begin{lem}\label{2.4}
If $m \equiv 0\pmod{16}$, then the graph $K_m \Box K_n$ has an $L_{8}$-decomposition.
\end{lem} 
\begin{proof}
Let $m \equiv 0\pmod{16}$  and $n \equiv  0, 2, 4, 6\pmod8$, then the proof follows from Lemma \ref{2.3}. Let $m = 0\pmod{16}$ and consider four cases for odd $n$. \\
\textit{\textbf{Case (1)}} $n \equiv \ 1 \pmod8$. \\
Let $m = 16s$ and $n = 8t+1$ for some $s, t > 0$. Then we can write $K_{m} \Box K_{n}$ as $s(t-1)\left(K_{16} \Box K_{8}\right)$ $\oplus$ $s\left(K_{16} \Box K_{9}\right)$ $\oplus$ $8s(t-1)(t-2)K_{8,8}$ $\oplus$  $\frac{s(8t+1)(s-1)}{2}K_{16,16}$ $\oplus$ $16s(t-1)K_{8,9}$. An $L_{8}$-decomposition of $K_{16} \Box K_{9}$ is shown in Appendix $4.2.1$ and the $L_8$-decomposition of the remaining graphs follows from Lemma \ref{2.2} and Theorem \ref{1.4}. \\
\textit{\textbf{Case (2)}} $n \equiv \ 3 \pmod8$. \\
Let $m = 16s$ and $n = 8t+3$ for some $s, t > 0$. Then we can write $K_{m} \Box K_{n}$ as $s(t-1)\left(K_{16} \Box K_{8}\right)$ $\oplus$ $s\left(K_{16} \Box K_{11}\right)$ $\oplus$ $8s(t-1)(t-2)K_{8,8}$ $\oplus$  $\frac{s(8t+3)(s-1)}{2}K_{16,16}$ $\oplus$ $16s(t-1)K_{8,11}$. An $L_{8}$-decomposition of $K_{16} \Box K_{11}$ is shown in Appendix $4.2.2$ and the $L_8$-decomposition of the remaining graphs follows from Lemma \ref{2.2} and Theorem \ref{1.4}. \\
\textit{\textbf{Case (3)}} $n \equiv \ 5 \pmod8$. \\
Let $m = 16s$ and $n = 8t+5$ for some $s, t > 0$. Then we can write $K_{m} \Box K_{n}$ as $s(t-1)\left(K_{16} \Box K_{8}\right)$ $\oplus$ $s\left(K_{16} \Box K_{13}\right)$ $\oplus$ $8s(t-1)(t-2)K_{8,8}$ $\oplus$  $\frac{s(8t+5)(s-1)}{2}K_{16,16}$ $\oplus$ $16s(t-1)K_{8,13}$. An $L_{8}$-decomposition of $K_{16} \Box K_{13}$ is presented in Appendix $4.2.3$ and the $L_8$-decomposition of the remaining graphs follows from Lemma \ref{2.2} and Theorem \ref{1.4}. \\
\textit{\textbf{Case (4)}} $n \equiv \ 7 \pmod8$. \\
Let $m = 16s$ and $n = 8t+7$ for some $s, t > 0$. Then we can write $K_{m} \Box K_{n}$ as $st\left(K_{16} \Box K_{8}\right)$ $\oplus$ $s\left(K_{16} \Box K_{7}\right)$ $\oplus$ $8st(t-1)K_{8,8}$ $\oplus$  $\frac{s(8t+7)(s-1)}{2}K_{16,16}$ $\oplus$ $16stK_{8,7}$. An $L_{8}$-decomposition of $K_{16} \Box K_{7}$ is presented in Appendix $4.2.4$ and the $L_8$-decomposition of the remaining graphs follows from Lemma \ref{2.2} and Theorem \ref{1.4}. \\
Hence the graph $K_m \Box K_n$ has the desired decomposition.
\end{proof}
\begin{lem}\label{2.5}
If $m \equiv 1\pmod{16}$ and $n \equiv 1\pmod16$, then the graph $K_m \Box K_n$ has an $L_{8}$-decomposition.
\end{lem} 
\begin{proof}
Let $m = 16s+1$ and $n = 16t+1$ for some $s, t > 0$. Then we can write $K_{m} \Box K_{n}$ as $(16t+1)K_{16s+1}$ $\oplus$ $(16s+1)K_{16t+1}$. By Theorem $\ref{1.2}$, the graph $K_{m} \Box K_{n}$ has the desired decomposition. 
\end{proof}
\begin{lem}\label{2.6}
If $m \equiv 15\pmod{16}$ and $n \equiv 3\pmod{16}$, then the graph $K_m \Box K_n$ has an $L_{8}$-decomposition.
\end{lem} 
\begin{proof}
Let $m = 16s+15$ and $n = 16t+3$ for some $s, t > 0$. We can write $K_{m} \Box K_{n}$ as $(16t+3)K_{16s+15}$ $\oplus$ $(16s+15)K_{16t+1}$. Now the first $16t$ columns can be viewed as $K_{16s} \ \oplus \ K_{15} \oplus sK_{16,15}$ and the first $16s$ rows can be viewed as $K_{16(t-1)} \ \oplus K_{19} \ \oplus (t-1)K_{16,19}$. Then $K_{16s}(= sK_{16} \ \oplus \ \frac{s(s-1)}{2}K_{16,16})$, $K_{16(t-1)}(= (t-1)K_{16} \oplus \frac{(t-1)(t-2)}{2}K_{16,16})$, $K_{16,15}$ and $K_{16,19}$ has an $L_{8}$-decomposition by Lemma \ref{1.3} and Theorem \ref{1.4}. The graph $K_{19}$ can be viewed as $K_{19} \backslash K_{3} \ \oplus \ K_{3}$. The $L_8$-decomposition of $K_{19} \backslash K_{3}$ follows from Appendix $4.3.1$. Then $16s\left(K_{19} \backslash K_{3}\right)$ has an $L_{8}$-decomposition. Then the remaining graph can be viewed as $s(K_{16} \Box K_{3})$, $t(K_{15} \Box K_{16})$ and $K_{15} \Box K_{3}$. Hence the desired decomposition follows from Appendixes ${4.3.2}$, ${4.3.3}$ and Lemma \ref{2.4} 
\end{proof}
\begin{lem}\label{2.7}
If $m \equiv 13\pmod{16}$ and $n \equiv 5\pmod{16}$, then the graph $K_m \Box K_n$ has an $L_{8}$-decomposition.
\end{lem} 
\begin{proof}
Let $m = 16s+13$ and $n = 16t+5$ for some $s, t > 0$. Then we can write $K_{m} \Box K_{n}$ as $st\left(K_{16} \Box K_{16}\right)$ $\oplus$ $t\left(K_{16} \Box K_{13}\right)$ $\oplus$ $s\left(K_{16} \Box K_{5}\right)$ $\oplus$ $K_{13} \Box K_{5}$ $\oplus$ $\frac{t(t-1)(16s+13)+s(s-1)(16t+5)}{2}K_{16,16}$ $\oplus$ $s(16t+5)K_{16,13}$ $\oplus$ $t(16s+13)K_{16,5}$. An $L_{8}$-decomposition of $K_{16} \Box K_{5}$ and $K_{13} \Box K_5$ are presented in Appendixes $4.4.1$ and $4.4.2$, respectively and the $L_8$-decomposition of the remaining graphs follows from Lemma \ref{2.4} and Theorem \ref{1.4}. Hence the graph $K_m \Box K_n$ has the desired decomposition.
\end{proof}
\begin{lem}\label{2.8}
If $m \equiv 11\pmod{16}$ and $n \equiv 7\pmod{16}$, then the graph $K_m \Box K_n$ has an $L_{8}$-decomposition.
\end{lem} 
\begin{proof}
Let $m = 16s+11$ and $n = 16t+7$ for some $s, t > 0$. Then we can write $K_{m} \Box K_{n}$ as $st\left(K_{16} \Box K_{16}\right)$ $\oplus$ $t\left(K_{16} \Box K_{11}\right)$ $\oplus$ $s\left(K_{16} \Box K_{7}\right)$ $\oplus$ $K_{11} \Box K_{7}$ $\oplus$ $\frac{s(s-1)(16t+7)+t(t-1)(16s+11)}{2}K_{16,16}$ $\oplus$ $(16t+7)sK_{16,11}$ $\oplus$ $(16s+11)tK_{16,7}$. An $L_{8}$-decomposition of $K_{11} \Box K_{7}$ is presented in Appendix $4.5.1$ and the $L_8$-decomposition of the remaining graphs follows from Lemma \ref{2.4} and Theorem \ref{1.4}. Hence the graph $K_m \Box K_n$ has the desired decomposition.
\end{proof}
\begin{lem}\label{2.9}
If $m \equiv 9\pmod{16}$ and $n \equiv 9\pmod{16}$, then the graph $K_m \Box K_n$ has an $L_{8}$-decomposition.
\end{lem} 
\begin{proof}
Let $m = 16s+9$ and $n = 16t+9$ for some $s, t > 0$. Then we can write $K_{m} \Box K_{n}$ as $st\left(K_{16} \Box K_{16}\right)$ $\oplus$ $(s+t)\left(K_{16} \Box K_{9}\right)$ $\oplus$ $K_{9} \Box K_{9}$ $\oplus$ $\frac{s(s-1)(16t+9)+t(t-1)(16s+9)}{2}K_{16,16}$ $\oplus$ $\left[s(16t+9)+t(16s+9)\right]K_{16,9}$. An $L_{8}$-decomposition of $K_{11} \Box K_{7}$ is presented in Appendix $4.6.1$ and the $L_8$-decomposition of the remaining graphs follows from Lemma \ref{2.4} and Theorem \ref{1.4}. Hence the graph $K_m \Box K_n$ has the desired decomposition.
\end{proof}
\subsection{Main Theorem}
Combining the results from Lemma \ref{2.1} to Lemma \ref{2.9}, we get the following main results.
\begin{thm}
The graph $K_m \Box K_n$ admits an $L_{8}$- decomposition if and only if one of the following holds:
\begin{enumerate}
 	\item $m, n \equiv \ 0 \ (mod \ 4)$ 
	\item $m \equiv \ 0 \ (mod \ 8), \ n \equiv \ 0 \ (mod \ 2)$
	\item $m \equiv \ 0 \ (mod \ 16)$
	\item $m \equiv \ 1 \ (mod \ 16), \ n \equiv \ 1 \ (mod \ 16)$ 
	\item $m \equiv \ 15 \ (mod \ 16), \ n \equiv \ 3 \ (mod \ 16)$
	\item $m \equiv \ 13 \ (mod \ 16), \ n \equiv \ 5 \ (mod \ 16)$
	\item $m \equiv \ 11 \ (mod \ 16), \ n \equiv \ 7 \ (mod \ 16)$
	\item $m \equiv \ 9 \ (mod \ 16), \ n \equiv \ 9 \ (mod \ 16)$
\end{enumerate}
\end{thm}
\section{Acknowledgments:}
The first author thank the Department of Science and Technology, Government of India, New Delhi for its financial support through the Grant No.DST/INSPIRE Fellowship/2015/IF150211. The second author thank the University Grant Commission, Government of India, New Delhi for its support through the grant No.F.510/7/DRS-I/2016(SAP-I).
\section{Appendix}
\subsection{$L_8$-decomposition required for Lemma \ref{2.3}}
\subsubsection{An $L_{8}$- decomposition of $K_{8} \Box K_{10}$} 
{\small $
$ for $(j,k) = (1,4), (3,5), (5,10), (7,5), (8,10), (9,6)$.} 
\subsubsection{An $L_{8}$- decomposition of $K_{8} \Box K_{6}$} 
{\small $%
$.}
\subsection{$L_8$-decomposition required for Lemma \ref{2.4}}
\subsubsection{An $L_{8}$- decomposition of $K_{16} \Box K_{9}$} 
{\small $%
$.}
\subsubsection{An $L_{8}$- decomposition of $K_{16} \Box K_{11}$} 
{\small $%
$.}
\subsubsection{An $L_{8}$- decomposition of $K_{16} \Box K_{13}$}
{\small $%
$.}
\subsubsection{An $L_{8}$- decomposition of $K_{16} \Box K_{7}$}
{\small $%
$.}
\subsection{$L_8$-decomposition required for Lemma \ref{2.6}}
\subsubsection{An $L_{8}$- decomposition of $K_{19} \backslash K_{3}$} 
Let $V(K_{19}) = \left\{x_{1}, x_{2}, . . . , x_{19}\right\}$ and the $K_{3}$ be $(x_{13}x_{15}x_{17})$. \\
{\small $%
$.}
\subsubsection{An $L_{8}$- decomposition of $K_{16} \Box K_{3}$}
{\small $%
$.}
\subsubsection{An $L_{8}$- decomposition of $K_{15} \Box K_{3}$}
{\small $%
$.}
\subsection{$L_8$-decomposition required for Lemma \ref{2.7}}
\subsubsection{An $L_{8}$- decomposition of $K_{16} \Box K_{5}$}
{\small $%
$.} 
\subsubsection{An $L_{8}$- decomposition of $K_{13} \Box K_{5}$}
{\small $%
$.}
\subsection{$L_8$-decomposition required for Lemma \ref{2.8}}
\subsubsection{An $L_{8}$- decomposition of $K_{11} \Box K_{7}$}
{\small $%
$.}
\subsection{$L_8$-decomposition required for Lemma \ref{2.9}}
\subsubsection{An $L_{8}$- decomposition of $K_{9} \Box K_{9}$}
{\small $%
$.}

\end{document}